\documentclass[12pt]{article}
\usepackage{enumerate,amsfonts,amsmath,amssymb,amsthm,graphicx,color}

\newtheorem {rema}{Remark}
\newtheorem {dfn}{Definition}
\newtheorem {lemma}{Lemma}

\newtheorem {thm}{Theorem}
\def\E{{\mathbb{E\,}}}

\def\P{{\mathbb{P}}}
\def\R{{\mathbb{R}}}

\def\F{{\cal{F}}}

\def\|{\,|\,}

\def\bn#1\en{\begin{align*}#1\end{align*}}
\def\bnn#1\enn{\begin{align}#1\end{align}}
\def\eps{\epsilon}
\title{On the generalization of the GMS evolutionary model}
\author{Skevi Michael and Stanislav Volkov}
\begin{document}
\maketitle
\begin{abstract}
We study a generalization of the evolution model proposed by Guiol, Machado and Schinazi (2010). In our model, at each moment of time a random number of species is either born or removed from the system; the species to be removed are those with the lower fitnesses, fitnesses being some numbers in $[0,1]$. We show that under some conditions, a set of species approaches (in some sense) a sample from a uniform distribution on $[f,1]$ for some $f\in [0,1)$, and that the total number of species forms a recurrent process in most other cases.
\end{abstract}

\section{Bak--Sneppen and Guiol--Machado--Schinazi models}

Over the last years the modeling of biological evolution has received a lot of attention in literature. Many models were proposed to explain and understand how nature works. A question that is a common reference to most research done over this field is why some species survive while others in the same ecosystem go extinct.

One of the models that were proposed for this purpose is the Bak-Sneppen model (BS) which was introduced by Per Bak and Kim Sneppen in 1993. The basic idea of their work was to build a model in which there exists a criterion that would represent the strength or resistance of each vertex in the ecosystem. This criterion is called {\it fitness}. The fitness of a vertex is usually related to its genetic code. The initial idea of their model was that the vertex with the weakest fitness is replaced by a new one. However, this leads to no interactions between the vertices and hence the model did not have receive much interest either from the biological or the mathematical point of view. To include the interaction factor in their model, they suggested that a ``weak'' vertex when leaving the system will also affect vertices that are connected with it and they will be removed from the system as well.

In particular, the BS model consists of an ``ecosystem''  that contains a (fixed) number $N$ of vertices which are located on the circumference of a circle. A quantity between $0$ and $1$ is assigned to each vertex, and it represents its fitness. At each time step the vertex with the lowest fitness is replaced by another one with a random fitness in the interval $[0,1]$. At the same time its two neighbours are also replaced by two other vertices with random fitnesses in $[0,1]$. This way no vertex can secure its survival no matter how ``strong'' its fitness is.

In the 1993 paper Bak and Sneppen showed that their model had the property of self-organized criticality and punctuated equilibrium. In the later years some more interesting results were proved, for example Meester and Znamenski (2003, 2004) studied the limit behaviour of the fitnesses, including a discrete version of the model, and they showed that the mean fitness is less than $1$ for the discrete case, which confirmed that the behaviour is indeed nontrivial. This was also supported by the simulations of the model which suggest that the limit distribution of the collection of fitnesses is uniform over the interval $[f,1]$ for some $f$ that is believed to be close to $2/3$.  However, so far one could not find a theoretical proof to confirm this behaviour.

In 2010 Guiol, Machado and Schinazi considered another stochastic model of evolution (we will refer to it as the GMS model) as an alternative for the BS model, since they believed that the setup of the BS model was a bit artificial and did not represent the nature well.
In the GMS model, the process starts wiht an empty subset of vertices of $[0,1]$. At each step, with probability $p$ a new vertex is born (birth case) and with probability $q=1-p$ one vertex is removed (death case). Each vertex that enters the system is assigned {\it a fitness value} which is an independent random variable uniformly distributed on $[0,1]$. In the death case the vertex with the lowest fitness is removed from the system. In Guiol et al (2010), it was proved that the set of vertices with fitness higher than a certain critical value $f_c=q/p$ will eventually approach a uniform distribution in the corresponding interval, with the error being of order less than $n^{1/2+\eps}$ for any $\eps>0$. Note that this mimics the behaviour which is expected to hold for the BS model.

There are two basic differences between the two models. In the GMS model the number of vertices in the system is random (not fixed) as in the BS model, which seems to be a more realistic approach to an evolutionary model. The second difference is that in the GMS model only the weakest vertex is removed at each time, hence there is no interaction among the vertices of the ecosystem. This means that a ``strong'' vertex is more likely to survive in the GMS model than in the BS model.

Recently, there were some finer results for the GMS model by Ben-Ari et al (2011), which included a $\log\log n$ correction term. Guiol et al in (2011) also discovered a link between the survival time in an evolution model and the Bessel distributions.

The Guiol et al (2010) paper motivated us to consider an extension of the GMS model, in which both the number of newborn and taken away vertices is random. Thus it makes the model even more realistic in expressing nature, as well as providing us with some non-trivial mathematical challenges. In Section~\ref{sec2} we assume that the number of deaths is a bounded random variable and obtain the results similar to those in Guiol et al (2010); this assumption is removed in Section~\ref{sec3} where we study the most general case.


\section{Multiple random births and deaths at each step}\label{sec2}

In our paper we will assume that at each step the numbers of vertices being born or taken away are random.
Namely, suppose that $X$ and $Z$ are two positive integer-valued random variables, $X_n$ ($Z_n$ resp.) are i.i.d.\ random variables with the distribution of $X$ ($Z$ resp.) and $X_n$'s and $Z_n$'s are all independent. Fix $p\in (0,1)$ and set $q=1-p$. At time $n$, the state of the system is a finite subset $T_n$ of vertices in $[0,1]$. By fitness of the vertex we understand its location on the segment $[0,1]$. Note that this setup covers the GMS model if we set $X\equiv Z\equiv 1$.

The system starts with an empty set, $T_0=\emptyset$. At time $n$, with probability $p$ we generate $Z_n$ new vertices, each having a fitness uniformly distributed over $[0,1]$ independently of each other and of anything else, so that $|T_{n+1}|=|T_n|+Z_n$; otherwise with probability $q=1-p$ we remove $X_n$ vertices with the smallest $X_n$ fitnesses, with the agreement that if there are less than $X_n$ vertices in the system, the system becomes empty again; as a result, $|T_{n+1}|=\max\{|T_n|-X_n,0\}$ here. Under some assumptions on the distributions of $X$ and $Z$ we will derive the results for the long-term behaviour of the system.

First, for some constant $f\in(0,1)$ define  $L_n$, $R_n$ and $R'_n$ as follows:
\begin{itemize}
\item[] $L_n :$ set of vertices alive in the system at time $n$ whose fitnesses lie in $[0, f)$
\item[] $R_n :$ set of vertices alive in the system at time $n$ whose fitnesses lie in $[f, 1]$
\item[] $R'_n :$ set of vertices that were born in the system from time $0$ to $n$ and were assigned a fitness in $[f, 1]$.
\end{itemize}
Obviously, $R_n\subseteq R'_n$.

\begin{dfn}
Suppose that $A_1\subseteq A_2\subseteq A_3\dots$ is an infinite sequence of sets, each consisting of a finite number of points in $\R$. We say that {\em $A_n$ approaches a random sample from distribution $F$} if, with probability $1$, there exists another sequence of sets $B_1\subseteq B_2\subseteq B_3\dots$ such that (i) each of these sets is a finite collection of i.i.d.\ random variables with the common distribution $F$; (ii) $|B_n|\to\infty $ as $n\to \infty$; and (iii) $|A_n \Delta B_n|=o(|B_n|)$ as $n\to\infty$. Here $A \Delta B=(A\setminus B) \cup (B\setminus A)$.
\end{dfn}
Let
\bnn\label{eqpc}
 p_c=\frac{\mu_X}{\mu_X+\mu_Z}.
\enn

\begin{thm}\label{emth2}
Assume that there is an (integer) constant  $M>0$ such that $X \leq M$ a.s., and  $\E(Z^2)<\infty$.
Let $\mu_X=\E(X)$ and $\mu_Z=\E (Z)$. Also suppose that $p\in (p_c,1)$ and let
\bnn\label{eqf}
 f=\frac{q}{p} \ \frac{\mu_X}{\mu_Z}\in (0,1).
\enn
Then, for every $\eps > 0$, there are $n_0\in \mathbb{N}$ and $C>0$ such that
$$
0 \leq |R'_n|-|R_n| \leq C \ n^{\frac{1}{2}+\eps} \quad \textrm{for}\ n \geq n_0 .
$$
Moreover, $T_n$ approaches a random sample from $U[f,1]$.
\end{thm}
\begin{proof}
The general skeleton of the proof is similar to that in~\cite{GMS}, although our model requires a deeper analysis.

First, look at those times when $|L_n| \geq M$. In the ``death'' event, $R_n$ is unaffected and all vertices will be removed from the complementary set $L_n$. Hence, for those times
\bnn\label{eqexpW}
\E \left( |L_{n+1}|-|L_n|\| \F_n\right)&=\E [W_n]  = p f \mu_Z -q \mu_X  =  0
\enn
where $\F_n$ is the sigma-algebra generated by the process by time $n$, and the distribution of random variable $W_n$ is given by
\bnn\label{eqWdis}
W_n=\left\{\begin{array}{ll}
Binomial(Z_n,f)& \text{ with probability }p,\\
-X_n& \text{ with probability }q.
\end{array}\right.
\enn
On the other hand, it is at the times when $|L_n|<M$ some vertices may be taken away from the set of $R_n$, resulting in $-M\le |R_{n+1}|-|R_n|\le 0$.

Define $t_n$ to be
$$
t_n= |\{1\leq k \leq n : |L_k| < M\}|
$$
the number of those ``bad'' times. We will show that $t_n$ is of order smaller than $n$.
Let
$$
k_n = \left| \left\{ 1\leq k \leq n : |L_{k-1}| \geq M \ \textrm{and}\ |L_k|<M \right\}\right|
$$
For any $\mu>0$
\bnn\label{eqb12}
 \P   \left( t_n > 2\mu n^{\frac{1}{2}+\eps} \right) \leq \P   \left(t_n>2\mu n^{\frac{1}{2}+\eps} ; k_n < n^{\frac{1}{2}+\eps} \right)+ \P   \left(k_n \geq n^{\frac{1}{2}+\eps} \right)=(I)+(II).
\enn
First, we want to choose an appropriate $\mu$ and hence to get an upper bound on $(I)$.
Set $E_1=0$  and for $i=1,2,\dots$ recursively define
\bn
G_i&=\min\{k> E_i: |L_k|\ge M\},\\
E_{i+1}&=\min\{k> G_{i}: |L_k|< M\}.
\en
Then
\bn
\{0\leq k \leq n : |L_k| < M\}&=\{0,1,2,\dots,n\}\bigcap \left(\bigcup_{i=1}^{\infty} [E_i,G_i)\right)\\
\text{ and } \max\{i:\ E_i\le n\}&=k_n+1.
\en

Let $\lfloor \cdot \rfloor$ denote the integer part of a number.
Observe that $G_i-E_i$ are stochastically smaller than i.i.d.\ non-negative random variables $\xi_i$ with the distribution given by
\bn
\P(\xi_i\ge m)=\left(1-(pf)^M\right)^{\lfloor m/M\rfloor}, \ m=0,1,2,\dots
\en
since for $|L_k|$ to reach $M$, even starting from $0$, it suffices to have $M$ consecutive birth events in which at least one of the new particles is located in $[0,f]$.  Let $\mu= \E \xi_i<\infty$ and
$$
m_n=\lfloor n^{1/2+\eps}\rfloor.$$
Then
\bn
(I) &=  \P   \left(t_n>2\mu n^{\frac{1}{2}+\eps} ; k_n < n^{\frac{1}{2}+\eps} \right) \leq
\P   \left( \sum_{i=1}^{k_n+1} [G_i-E_i] > 2\mu n^{\frac{1}{2}+\eps}\ ,k_n<n^{\frac12+\eps}  \right) \\
&\leq
\P   \left( \sum_{i=1}^{m_n+1} [G_i-E_i] > 2\mu n^{\frac{1}{2}+\eps}\right) \leq
\P   \left( \sum_{i=1}^{m_n+1} \xi_i > 2\mu n^{\frac{1}{2}+\eps}\right)
\leq
\P   \left( \sum_{i=1}^{m_n+1} \xi_i > 2\mu m_n\right).
\en
At this point we will use a large deviation estimate which follows immediately from Lemma~9.4 in Chapter~1.9 of Durrett (1996):

\begin{lemma}\label{lardevdur}
Let ${\bf X}_1, {\bf X}_2,\ldots, {\bf X}_n$ be an i.i.d.\ sequence of random variables with $\mu:=\E {\bf X}_i$ and $\phi(\vartheta):=\E (e^{\vartheta {\bf X}_i})<\infty$ for some positive $\vartheta$. Let $\kappa(\vartheta)=\log \phi(\vartheta)$ and  $S_n={\bf X}_1+{\bf X}_2+\ldots+{\bf X}_n$. Then for $\alpha>\mu$,
\bn
\P  \left(S_n\geq n\alpha\right)\leq \exp\{-n\left(\alpha\vartheta-\kappa(\vartheta)\right)\}.
\en
Moreover, for $\vartheta$ small we have  $\alpha\vartheta-\kappa(\vartheta)>0$.
\end{lemma}
By applying this lemma to the sequence of $\xi_i$ (note that $\E (e^{\vartheta \xi_i})<\infty$ for sufficiently small $\vartheta$  due to the fact that $\xi_i$ is a linear transformation of an exponential random variable), we obtain that there exists $\theta >0$ such that

\bnn\label{eqb1}
(I)\le \P   \left( \sum_{i=1}^{m_n+1} \xi_i > 2\mu m_n \right) \leq \exp\left(-\theta n^{1/2+\eps}\right)
\text{ for all } n.
\enn

Next, we want to get an upper bound on $(II)$.
We have
\bn
(II)=\P\left(k_n\ge n^{\frac12+\eps}\right)\le \P\left(\bigcap_{i=1}^{\ m_n}\left\{E_{i+1}-G_i<n\right\}\right).
\en
At the same time, for each $i\ge 1$, $E_{i+1}-G_i$ is stochastically smaller than a random variable $\tau_i\ge 1$,  where $\tau_i$'s are independent and each having the distribution of
\bn
 \tau=\min\{j\ge 1:\ W_1+W_2+\dots +W_j<0\}.
\en
Here $W_k$ are i.i.d.\ random variables with the distribution given by~(\ref{eqWdis}).
To estimate $\P(\tau<n)$ we use a result from the general theory of  random walks given in Feller (1966) volume 2 (Theorem 1.a in Chapter XII.8 and Theorem 1 in Chapter XVIII.5 respectively).

\begin{thm}\label{feller1} Let ${\bf X}_1, {\bf X}_2,\ldots, {\bf X}_n$ be an i.i.d, sequence of random variables with distribution $F$ such that $0<F(0)<1$. We define the sums $S_i, i \in \{0,1,\ldots,n\}$ so that $S_0=0$ and $S_n={\bf X}_1+{\bf X}_2+\ldots+{\bf X}_n$ and let
\bn
K_n &=\{0\leq k\leq n :S_k>S_0, S_k >S_1,\ldots, S_k>S_{k-1},S_k \geq  S_{k+1}, \ldots S_k  \geq S_n\}\\
&=\min\{j:\ S_j=\max_{i\in [0,n]} S_i \}.
\en
If the series
\bnn\label{serfel}
\sum_{k=1}^{\infty}\frac 1n \left(\P  \left(S_n>0\right)-\frac{1}{2}\right)
\enn
converges then for $0\leq k\leq n$,
\bn
\P  \left(K_n=k\right) \sim \binom{2k}{k} \binom{2n-2k}{n-k} \frac{1}{2^{2n}}
\en
where $a_n \sim b_n$ means that $\lim _{n\to \infty}{a_n}/{b_n}=1.$
\end{thm}

\begin{thm}\label{feller2}
Consider the notation of Theorem \ref{feller1} and suppose that its conditions hold. If $\E {\bf X}_1=0$ and $\E ({\bf X}_1^2)=\sigma^2<\infty$, then the series (\ref{serfel}) is at least conditionally convergent.
\end{thm}

Now we can apply Theorems~\ref{feller1} and~\ref{feller2} by setting ${\bf X}_i=-W_i$ since
$\E W_i=0$ due to~(\ref{eqexpW}), and $\E W_i^2<\infty$ due to the fact that $X$ is bounded and $\E(Z^2)<\infty$.
Consequently
\bn
 \P(\tau\ge n)&=\P({\bf X}_1\le 0,\ {\bf X}_1+{\bf X}_2\le 0,\ \dots,\ {\bf X}_1+\dots+{\bf X}_n\le 0)
 \\
 &=\P(K_n=0)\sim {2n \choose n} \frac 1{2^{2n}}=\frac {(2n)!}{(n!)^2 \, 2^{2n}}\sim \frac 1{\sqrt{\pi n}}
\en
where we used  Stirling's formula in the last equation.
Combining the above calculations we have that
\bn
(II)&\le (\P(\tau< n))^{m_n}=(1-\P(\tau\ge n))^{m_n}= \left(1-\frac {1+o(1)}{\sqrt{\pi n}}\right)^{m_n}
= \left[\exp\left(-\frac {1+o(1)}{\sqrt{\pi n}}\right)\right]^{m_n}.
\en
Therefore, there is a constant $\alpha>0$ such that for all large $n$
\bnn\label{eqb2}
 (II)  \leq \  e^{-\alpha n^{\eps}}.
\enn
Hence, plugging~(\ref{eqb1}) and~(\ref{eqb2}) into~(\ref{eqb12}) we obtain
\bn
\sum_{n=1}^{\infty}\P\left(t_n>2\mu n^{\frac12+\eps}\right)<\infty
\en
and therefore by the Borel-Cantelli lemma a.s.\ there is an $n_0$ such that $t_n\le 2\mu n^{\frac12+\eps}$ for all $n\ge n_0$.

Since $|R_{n+1}'|-|R_n'|= |R_{n+1}|-|R_n|+\Delta$, where
$\Delta= 0$ if  $|L_n|\ge M$ and $\Delta \in \{0,1,2,\dots,M\}$ if  $|L_n|< M$, we have

\bn
|R'_n|-|R_n| &= \sum_{k=0}^{n-1} \left(|R'_{k+1}| - |R'_{k}|\right)- \sum_{k=0}^{n-1}\left(|R_{k+1}| - |R_{k}|\right)
 \\
 &=\sum_{k=0}^{n-1} \left[(|R'_{k+1}| - |R'_{k}|)- (|R_{k+1}| - |R_{k}|)\right] 1_{\{|L_k|< M\}}
 \le M t_n\le 2\mu M n^{1/2+\eps}
\en
for all $n\ge n_0$.

Finally, to yield the final statement of Theorem~\ref{emth2}, observe that $R'_n$ is a collection of i.i.d.\ random variables from $U[f,1]$, and
$$
|T_n\Delta R_n'|= |L_n|+|R'_n\setminus R_n|.
$$
On one hand, we have $|R'_n\setminus R_n|\le C n^{1/2+\epsilon}$ for large $n$. On the other hand,
$$
\lim_{n\to\infty} \frac{|R_n'|}n= p\mu_Z (1-f)\text{\ \ \ a.s.\ \ and \ \ }
\limsup_{n\to\infty} \frac{|L_n|}n\le \E W=0
\text{\ \ \ a.s.}
$$
(see (\ref{eqexpW})\,) by the strong law. Therefore $|R_n'|\to \infty$ a.s.\ and
\bn
\lim_{n\to\infty} \frac{|T_n\Delta R_n'|}{|R'_n|}
=\lim_{n\to\infty} \frac{|L_n|/n}{|R'_n|/n}+\lim_{n\to\infty} \frac{|R'_n\setminus R_n|/n}{|R'_n|/n}=0
\text{\ \ \ a.s.}
\en
Thus $T_n$ approaches a random sample from $U[f,1]$.
\end{proof}

\section{Number of deaths unbounded}\label{sec3}
In this section we will generalize the model to the case when $X$ is not necessarily bounded. We will show that finiteness of $\E X$ is essentially a necessary and sufficient condition for $T_n$ to approach a random sample from a uniform distribution.

First, we will prove a simple fact about the expectation of a non-negative integer random variables.
\begin{lemma}\label{ellem}
Let $X$ be a non-negative integer random variable. Then $\E X<\infty$ if and only if for every $c>0$
\bn
\sum_{n=1}^{\infty} \P(X\ge cn)<\infty.
\en
\end{lemma}
\begin{proof}
Let $\mu=\E X=\sum_{n=1}^{\infty} \P(X\ge n)\in[0,\infty]$. First, suppose that $c\le 1$. Then there exists an integer $m>0$ such that $1/m<c$. We have
\bnn\label{eqlem1}
 \mu=\sum_{n=1}^{\infty} \P(X\ge n)\le   \sum_{n=1}^{\infty} \P(X\ge cn)
  \le   \sum_{n=1}^{\infty} \P(X\ge n/m)=m  \sum_{k=1}^{\infty} \P(X\ge k)=m \mu.
\enn
Now if $c>1$, there exists and integer $m>c$. Then
\bnn\label{eqlem2}
 \mu&=\sum_{n=1}^{\infty} \P(X\ge n) \ge -1+\sum_{n=0}^{\infty} \P(X\ge cn)\ge  -1+\sum_{n=0}^{\infty} \P(X\ge nm)
  \nonumber \\
 &\ge -1+\sum_{n=0}^{\infty} \frac 1m \left[ \sum_{k=0}^{m-1} \P(X\ge nm+k) \right]
 =-1+\frac 1m  \sum_{k=0}^{\infty} \P(X\ge k)=\frac{1+\mu}m-1.
\enn
Together, (\ref{eqlem1}) and (\ref{eqlem2}) yield the statement of the Lemma.
\end{proof}

Recall that $T_n$ is the set of species alive in the system at time $n$, so in case that we have a death event,
$$
|T_{n+1}|=\max\left\{0, |T_n|-X_n\right\}.
$$
Moreover, assuming $p>p_c$, for every $\eps \in [0,1-f)$ where $p_c$ is given by~(\ref{eqpc}) and $f$ is the same as in~(\ref{eqf}) define
\bn
L_n^{\eps}:=T_n \cap \left[0, f+\eps\right) \text{ and } R_n^{\eps}:=T_n \cap \left[f+\eps, 1\right] \\
\en
Note that $L_n^0=L_n$ and $R_n^0=R_n$. Also, define $A_n^{\eps}$ as follows:
$$
A_n^{\eps}=\left\{\text{at time } n \text{ we kill {\it all} vertices in}\ L_n^{\eps}\right\}=\{L_{n+1}^{\eps}=\emptyset\}.
$$

\begin{lemma}\label{Anfin}
Suppose $\mu_Z=\E Z<\infty$, $\mu_X=\E X<\infty$  and $p>p_c$. Then, with probability $1$, $A_n^{\eps}$ occurs finitely often.
\end{lemma}

\begin{proof}
First note that
 $$
 |L_{n+1}^{\eps}| = \begin{cases}
 |L_n^{\eps}|+ Y_n, \textrm{where}\ Y_n \sim Binomial(Z_n,f+\eps),& \mbox{with probability }  p,\\
 \max\{|L_n^{\eps}|-X_n,0\}, & \mbox{with probability }  q.
 \end{cases}
 $$
Now, let $Q_0=0$ and define $Q_i$ recursively as
$$
Q_{n+1} = Q_n+\begin{cases}  Binomial(Z_n,f+\eps), & \mbox{with probability } p,\\
 -X_n, & \mbox{with probability } q.
 \end{cases}
$$
Thus $Q_n$ can be coupled with $|L_n^{\eps}|$ in such a way that $|L_n^{\eps}|\ge Q_n$ for all $n$.
On the other hand, $Q_n$ can be written as a sum of $n$ i.i.d.\ random variables each with expectation
$2\delta:=p\mu_Z(f+\eps)-q\mu_X=p\mu_Z\eps>0$.
By the strong law of large numbers we have
$$
\lim_{n\to\infty} \frac{Q_n}{n}=2\delta \text{\ \ \ a.s.}
$$
Hence
$$
\liminf_{n\to  \infty} \frac{|L_n^{\eps}|}{n} \geq 2\delta  \text{\ \ \ a.s.}
$$
which yields that with probability $1$ there exists a time $N_0\in \mathbb{N}$ such that
$|L_n^{\eps}|>n \delta$ for all $n\geq N_0$. Next we calculate the probability that $A_n^{\eps}$ occurs:
\bn
\P  \left(A_n^{\eps}\right) = q  \P  \left(X_n\ge |L_n^{\eps}|\right)
\le q  \P  \left(X \ge n \delta\right) \text{ for } n\geq N_0.
\en
Consequently, by Lemma~\ref{ellem}, since $\E X<\infty$, we have
\bn
\sum_{n=1}^{\infty} \P  \left(A_n^{\eps}\right)
 \le N_0+q  \sum_{n=1}^{\infty} \P  \left(X \ge n \delta\right) <\infty
\en
and so by the Borel-Cantelli lemma, we have that $A_n^{\eps}$ occurs finitely often with probability $1$.
\end{proof}

Recall that $\mu_Z=\E Z$ and $\mu_X=\E X$.
\begin{thm}\label{emth3}
The following is true.
\begin{itemize}
\item[(a)]
Suppose $\mu_Z=\infty$ and $\mu_X<\infty$. Then $T_n$ approaches a random sample from $U[0,1]$.
\item[(b)]
Suppose $\mu_Z<\infty$. If $\mu_X<\infty$ and $p>p_c$ then  $T_n$ approaches a random sample from $U[f,1]$ where $f$ is given by~(\ref{eqf}).
\item[(c)]
Suppose $\mu_Z<\infty$. If $\mu_X<\infty$ and $p<p_c$, or $\mu_X=\infty$, then  $T_n=\emptyset$ for infinitely many $n$.
\end{itemize}
\end{thm}
\begin{rema}
Theorem~\ref{emth3} leaves some gaps: it covers neither the critical case when $\mu_X,\mu_Z<\infty$ and $p=p_c$, nor the general case where both $\mu_X=\mu_Z=\infty$.
\end{rema}
\begin{proof}
(a) Let $B_n$ be the set of all the particles born in the system by time $n$ and let $D_n$ be the set of particles removed from the system by time $n$; therefore $B_n$ is a collection of i.i.d.\ $U[0,1]$ random variables and $T_n=B_n\setminus D_n$. Since at each time $n$ we remove from the system {\em at most} $X_n$ particles, by the strong law of large numbers we have
\bn
\limsup_{n\to\infty} \frac{|D_n|}n & \le q \mu_X <\infty \text{\ \ \ a.s., and}&\\
\lim_{n\to\infty} \frac{|B_n|}n & =p\mu_Z  =\infty \text{\ \ \ a.s.}&
\en
Therefore,
\bn
\limsup_{n\to\infty} \frac{|T_n\Delta B_n|}{|B_n|}
=\limsup_{n\to\infty} \frac{|D_n|}{|B_n|}=\limsup_{n\to\infty} \frac{|D_n|/n}{|B_n|/n}=0 \text{\ \ \ a.s.}
\en
which yields the desired conclusion.

\vskip 5mm

(b) Assume $\mu_X<\infty$ and $p\in(p_c,1)$. Recall that $R'_n $ denotes a set of vertices that were born in the system up to the time $n$ that were assigned a fitness in $[f, 1]$; thus $R'_n$ is a collection of i.i.d. $U[f,1]$ random variables. Moreover, as in the proof of Theorem~\ref{emth2}, $|R'_n|\to \infty$. Fix an $\eps>0$ and observe that $R_n^{\eps}\subseteq R_n\subseteq R'_n$. According to Lemma~\ref{Anfin}, there will be a time $N_1$ such that events $A_n^{\eps}$ do not occur for $n\ge N_1$. This implies that no vertices are taken away from $[f+\eps,1]$ for those $n$, and as a result
$$
\sup_{n} \left|(R'_n\setminus R_n^{\eps})\cap[f+\eps,1]\right|<\infty \text{\ \ \ a.s.}
$$
On the other hand, by the strong law we have $|R'_n\cap[f,f+\eps,1)|/n\to \eps\, p\, \mu_Z$ a.s., therefore
$$
0\le \limsup_{n\to\infty} \frac{|R'_n\setminus R_n|}n
\le
\limsup_{n\to\infty} \frac{|R'_n\setminus R_n^{\eps}|}n\le  \eps \, p\, \mu_Z  \text{\ \ \  a.s.}
$$
Since $\eps>0$ is arbitrary, we conclude
$$
\lim_{n\to\infty} \frac{|R'_n\setminus R_n|}n=0\text{\ \ \  a.s.}
$$
From the end of the proof of Theorem~\ref{emth2} we have
$|R_n'|/n\to p\mu_Z (1-f)$ a.s.
and
$|L_n|/n\to 0$ a.s.
therefore
$$
\limsup_{n\to\infty} \frac{|R'_n\Delta T_n|}{|R'_n|}
= \limsup_{n\to\infty} \frac{|R'_n\setminus R_n|+|L_n|}{|R'_n|}
\le
\limsup_{n\to\infty} \frac{|R'_n\setminus R_n|/n}{|R'_n|/n}
+
\limsup_{n\to\infty} \frac{|L_n|/n}{|R'_n|/n}=0
 \text{\ \ \  a.s.}
$$
which proves that $T_n$ approaches a random sample from $U[f,1]$.

\vskip 5mm

(c)  Now suppose $\mu_X<\infty$ but $p<p_c$. Due to the renewal nature of the process, it is sufficient to demonstrate that there exist a.s.\ at least one $n\ge 1$ such that $T_n=\emptyset$.

Let $W_n$ be i.i.d.\ random variables with the distribution given by
\bn
W_n=\left\{\begin{array}{ll}
 Z_n& \text{ with probability }p,\\
 -X_n& \text{ with probability }q.
\end{array}\right.
\en
Let $\tau=\inf\{n\ge 1:\ W_1+W_2+\dots+W_n\le 0\}$. Then $\tau+1$ has the same distribution as
$\inf\{n\ge 1:|T_n|=0\}$.  Observe that by the strong law
$$
\lim_{n\to\infty} \frac{W_1+\dots+W_n}n= \E W =p\mu_Z-q \mu_X=\mu_X \left(\frac p{p_c}-1\right)<0.
$$
Therefore we must have $\tau<\infty$  a.s.
\vskip 5mm

Finally, assume that $\mu_X=\infty$. By the strong law we have
$$
\limsup_{n\to\infty} \frac{|T_n|}n \le \lim_{n\to\infty} \frac{|B_n|}n = p\,\mu_Z \text{\ \ \ a.s.}
$$
therefore there exists $c>0$ and a positive integer $N_3$ such that $|T_n|\le cn$ for all $n\ge N_3$.
On the other hand, by Lemma~\ref{ellem},
$$
\sum_n \P(X_n\ge cn)=\sum_n \P(X\ge cn)=\infty,
$$
and since the events $\{X_n\ge cn\}$ are independent, by the second Borel-Cantelli Lemma there will be infinitely many $n$ for which $X_n\ge cn\ge |T_n|$ and hence $T_{n+1}=\emptyset$.
\end{proof}

\end{document}